\documentclass[11pt,letterpaper]{amsart}

\usepackage{amsmath}
\usepackage{amssymb} 
\usepackage[all,cmtip]{xy}
\usepackage{fancyvrb}
\usepackage{hyperref}

\usepackage{stmaryrd} 

\newcommand{\Z}{{\mathbb Z}} 
\newcommand{\Q}{{\mathbb Q}} 
\newcommand{\C}{{\mathbb C}} 
\newcommand{\F}{{\mathbb F}} 
\newcommand{\R}{{\mathbb R}} 
\newcommand{\D}{{\mathbb D}} 

\newcommand{\cO}{{\mathcal O}} 
\newcommand{\cF}{{\mathcal F}} 
 
\newcommand{\cE}{{\mathcal E}} 
\newcommand{\cG}{{\mathcal G}}

\DeclareMathOperator{\ord}{ord}

\newcommand{\gp}{{\mathfrak p}}

\newcommand {\llp} { \llparenthesis }
\newcommand {\rrp} { \rrparenthesis }

\newtheorem{thm}{Theorem}

\newtheorem{prop}{Proposition}

\theoremstyle{remark}

\newtheorem{rem}[]{Remark}

\newcommand{\itop}[2]{\genfrac {}{}{0pt}{3}{#1}{#2} }

\begin{document}

\author[P. Guerzhoy]{P. Guerzhoy}
\address{ 
Department of Mathematics,
University of Hawaii, 
2565 McCarthy Mall, 
Honolulu, HI,  96822-2273 
}
\email{pavel@math.hawaii.edu}

\title[]{Non-vanishing of a certain  quantity related to the $p$-adic coupling of mock modular forms with newforms}
\begin{abstract}
Several authors have recently proved results which express a cusp form as a $p$-adic limit of weakly holomorphic modular forms under repeated application of Atkin's $U$-operator.
Initially, these results had a deficiency: one could not rule out the possibility when a certain quantity vanishes and the final result fails to be true. 
Later on, Ahlgren and Samart  \cite{AS} found a method to prove the non-vanishing in question in the specific case considered by El-Guindy and Ono \cite{EGO}. Hanson and Jameson \cite{HJ} and (independently) Dicks \cite{Dicks} 
generalized this method to finitely many other cases. 

In this paper, we present a different approach which allows us to prove a similar non-vanishing result for an infinite family of similar cases. Our approach also allows us to return back to the original example considered by El-Guindy and Ono \cite{EGO}, where we calculate the (manifestly non-zero) quantity explicitly in terms of Morita's $p$-adic $\Gamma$-function.

\end{abstract}

\maketitle

\section{Introduction} \label{sec_intro}

Let $g \in S_2(N)$ be a primitive form (i.e. a normalized new cusp Hecke eigenform) of weight $2$ with trivial Nebentypus, on congruence subgroup $\Gamma_0(N)$ with Fourier expansion 
\[
g = \sum_{n \geq 1} b(n) q^n,
\] 
where $q=\exp(2 \pi i \tau)$ and $\Im (\tau) >0$. 
Assume that $b(n) \in \Z$ (and recall the normalization $b(1)=1$). 

Let $f$ be {\sl a harmonic Maass form which is  good for} $\overline{g(-\overline{\tau})}=g(\tau)$  (see \cite[Chapter 4]{BFOR} for the basic definitions related to the harmonic Maass forms and \cite[Definition 7.3]{BFOR} for the definition of being "good"; the existence of such $f$ follows from \cite[Proposition 5.1]{BOR} (note that $f \in H_{0, \infty}(\Gamma_0(N))$ in the notations of \cite{BOR}). Since the weight of $f$ is zero, the splitting $f=f^+ + f^-$ 
(see \cite[Lemma 4.3, Definition 4.4]{BFOR}) 
of $f$ into a {\sl mock modular form} $f^+$ whose shadow is $g$ and a non-holomorphic part $f^-$ becomes 
the splitting of a harmonic function $f$ into its meromorphic and antiholomorphic parts. 

Let 
\[
\cE_g = \sum_{n \geq 1} \frac{b(n)}{n} q^n
\]
be the Eichler integral associated with $g$. By \cite[Theorem 1.1]{GKO}, there exists a complex number $\alpha \in \C$ (note that in fact $\alpha \in \R$, but we do not pursue that in this article) such that the Fourier expansion
\begin{equation} \label{equ_def_alpha}
\cF_\alpha:= f^+ - \alpha \cE_g = \sum_{n \gg -\infty} \frac{d_\alpha(n)}{n} q^n \in \Q\llparenthesis q \rrparenthesis \hookrightarrow \Q_p \llparenthesis q \rrparenthesis
\end{equation}
has all rational coefficients which we consider as $p$-adic numbers (for any prime $p$). 
It is clear that the  choice of $\alpha$ is not unique; we fix an arbitrary $\alpha$ such that \eqref{equ_def_alpha} holds for the time being.  
We will make a specific choice of $\alpha$ in \eqref{equ_choice_of_alpha} later.

In order to describe the question addressed in this paper, we now recall some facts about $p$-adic coupling   of mock modular forms with newforms \cite[Section 7.4.2]{BFOR}.

We denote by $U_p=U$ Atkin's $U$-operator and by $V_p=V$ its one-sided inverse which act on formal power series as
\[
\left. \left( \sum_{n \gg -\infty} a(n) q^n \right) \right| U :=  \sum_{n \gg -\infty} a(pn) q^n,  
\]
\[
 \left. \left( \sum_{n \gg -\infty} a(n) q^n \right) \right| V := \sum_{n \gg -\infty} a(n) q^{pn}.
\]


For arbitrary $\beta,\delta \in \Q_p$, consider
\[
\cF_{\alpha,\beta,\delta} = \cF_\alpha - \beta\cE_g - \frac{\delta}{p}\cE_g|V = \sum_{n \gg -\infty} \frac{d_{\alpha,\beta,\delta}(n)}{n} q^n \in \Q_p\llparenthesis q \rrparenthesis .
\]

By the theory of harmonic Maass forms (see \cite[Theorem 5.9]{BFOR}), 

\[
D\cF_{\alpha,\beta,\delta} := \frac{1}{2\pi i} \frac{d}{d\tau} \cF_{\alpha,\beta,\delta} = q \frac{d}{dq}\cF_{\alpha,\beta,\delta} =  \sum_{n \gg -\infty} d_{\alpha,\beta,\delta}(n) q^n \in  \Q_p\llparenthesis q \rrparenthesis
\]
is a $q$-expansion of a weakly holomorphic modular form: since $D\cF_\alpha \in S_2^!(N)$,  we also have that $D\cF_{\alpha,\beta,\delta} \in S_2^!(pN) \otimes \Q_p$.

\begin{prop} \label{prop_coupling}

Let $p \nmid N$ be a prime such that $b(p) = 0$. There exist a unique $\beta \in \Q_p$ such that for all except exactly one $\delta \in \Q_p$ one has 
\[
\lim_{m \rightarrow \infty} \frac{\left. \left(D\cF_{\alpha,\beta,\delta} \right) \right|U^{2m+1}} {d_{\alpha,\beta,\delta}(p^{2m+1})} = g.
\]
\end{prop}

Proposition \ref{prop_coupling} is essentially a specialization to the case when the weight of the mock modular form is zero of 
\cite[Theorem 7.12(ii)]{BFOR}, \cite[Theorem 1.2(2)]{GKO}, \cite[Proposition 1.4]{BGK}. 
In this case, however, our setting is more general: we only assume that $b(p) = 0$ instead of assuming that $g$ has complex multiplication and the prime $p$ is inert in the imaginary quadratic  field. Note that the seminal result of Elkies \cite{Elkies} implies that there exist infinitely many primes $p$ such that $b(p)=0$.

Since our setting differs slightly from that in the cited literature, the meaning of the quantities $\alpha, \beta, \delta$, and $\gamma$ below is different from that in the cited articles (where it also fluctuates). 

\begin{proof}

For a formal power series $\sum c(n) q^n \in \Q_p\llparenthesis q \rrparenthesis$, define its $p$-adic order 
\[
\ord_p \left(\sum c(n) q^n\right)= \inf_n\{\ord_p(c(n))\}.
\]

By \cite[Proposition 5]{adele}, there exist (obviously unique assuming that $\alpha$ is fixed) $\beta, \gamma \in \Q_p$ such that
\begin{equation} \label{equ_except_gamma} 
\ord_p \left( \cF_{\alpha,\beta,\gamma} \right) = \ord_p \left( f^+-\alpha \cE_g - \beta \cE_g - \gamma \frac{1}{p} \cE_g|V \right) >-\infty.
\end{equation}
It follows, in particular, that 
\[
\ord_p(D(\cF_{\alpha,\beta,\gamma})|U^{2m+1}) > 2m+1-c ~~ \text{and} ~~                                           \ord_p(d_{\alpha,\beta,\gamma}(p^{2m+1})) > 2m+1 -c
\]
for some fixed $c \in \Z$. Our assumptions (i.e.  $g$ is Hecke eigenform, $p \nmid N$,  $b(p)=0$) imply that $b(p^{2m}) = (-p)^m$. 
Since
\[
d_{\alpha,\beta,\delta}(p^{2m+1}) = d_{\alpha,\beta,\gamma}(p^{2m+1}) - (\delta - \gamma) b( p^{2m}) =  d_{\alpha,\beta,\gamma}(p^{2m+1}) - (\delta - \gamma) (-p)^m,
\]
we can conclude that for $\delta \neq \gamma$
\[
\ord_p(d_{\alpha,\beta,\delta}(p^{2m+1}) ) = m + \ord_p( \delta - \gamma)
\] 
for $m$ large enough. Thus for   $\delta \neq \gamma$
\[
\lim_{m \rightarrow \infty} \frac{\left. \left(D\cF_{\alpha,\beta,\gamma} \right) \right|U^{2m+1}} {d_{\alpha,\beta,\delta}(p^{2m+1})} = 0.
\]
Since $\cF_{\alpha,\beta,\delta} = \cF_{\alpha,\beta,\gamma} - \frac{(\delta - \gamma)}{p} \cE_g|V$, and $D(\cE_g|V) = pg|V$, Proposition \ref{prop_coupling} now follows from 
$g|U^2 = - p g$.
\end{proof}

The above proof allows us to characterize the exceptional $\delta = \gamma$ value. 
Namely, that is exactly the unique $\gamma \in \Q_p$ such that there exists (a unique)  $\beta \in \Q_p$ which makes \eqref{equ_except_gamma} true. 
Note that in fact \cite[Proposition 5]{adele} does not require the condition $b(p)=0$ to guarantee the existence of $\beta, \gamma \in \Q_p$ such that 
\eqref{equ_except_gamma} holds. 

Note that, since $g$ is a Hecke eigenform and $b(p)=0$, the $q$-expansions of $\cE_g$ and $\cE_g|V$ are supported on disjoint sets of indexes 
(because $b(n) \neq 0$ implies $b(pn) =0$ for any $n \geq 1$). 
Thus, while $\beta \in \Q_p$ in \eqref{equ_except_gamma} certainly depends on the choice of $\alpha \in \R$, the quantity $\gamma \in \Q_p$ does not depend on that choice 
(as soon as that choice guarantees $\cF_\alpha \in \Q\llparenthesis q \rrparenthesis$).

In the literature,  statements similar to Proposition \ref{prop_coupling}  were formulated and proved with $\delta=0$ (under some additional conditions which guarantee that $\gamma \neq 0$). 
The question whether $\gamma = 0$ may happen at all was raised in 
\cite[Remark 2 to Theorem 1.3]{BGK} in a very general context, \cite[Remark 8]{GKO}, \cite[Remark to Theorem 1.1]{EGO} (in a slightly different form). 
A method to prove that $\gamma \neq 0$ for $N=32$ (answering the question as it is put in  \cite{EGO}) was developed in \cite{AS}. 
This method was later employed in \cite{HJ, Dicks, Tajima}  to prove that $\gamma \neq 0$ in all the (finitely many) cases when $g$ is a CM form and $\dim S_2(N) =1$. (This work also includes the example considered in \cite{GKO} where the weight of $g$ is $4$, which is not covered in the present paper.)

In this paper, we offer an alternative approach to this question which shows that $\gamma \neq 0$ is a more general fact which can be detached from one-dimensionality, complex multiplication, and $\eta$-quotient presentations of certain modular forms, all important for  variations and generalizations of the approach developed in \cite{AS}.


\begin{thm} \label{thm_gamma_modular}

In the notations and conventions described above, $\gamma \neq  0$  for all but possibly finitely many primes $p \nmid N$ such that $b(p) = 0$.

\end{thm}

Although Theorem \ref{thm_gamma_modular} provides for an abundance of examples, we keep illustrating our approach with the original setting of \cite{EGO, AS} throughout.
While  (a sharper version of) the specialization of Theorem \ref{thm_gamma_modular} to this special case is proved in \cite{AS}, our interpretations and results pertaining to this example are still interesting. 
We thus set
\[
g(\tau)=\eta^2(4\tau)\eta^2(8\tau) \in S_2(32)
\]

The condition $b(p)=0$ with $p \nmid 32$ translates to the primes $p$ being inert in $\Q(\sqrt{-1})$, namely $p \equiv 3 \pmod 4$.
In this case, we calculate $\gamma$ explicitly in terms of Morita's $p$-adic $\Gamma$-function which we denote by $\Gamma_p(\cdot)$.
Denote by $C(n)$ the $n$-th  Catalan number (see \eqref{eq_catalan} below).

\begin{thm} \label{thm_mu_32}
Let $g$ and $p$ be as above. Then 

\begin{multline*}
\gamma = \left( \frac{2}{p} \right) \lim_{m \rightarrow \infty} \frac{5C((p^{2m+1}+1)/4)}{3C((p^{2m} -1)/4)}  \\
= 8\left( \frac{2}{p} \right) \frac{\Gamma_p(1/2)}{\Gamma_p(1/4)^2} = 
\begin{cases} \left( \frac{2}{p} \right) (-1)^{(3+h(-p))/2} \frac{8}{\Gamma_p(1/4)^2} & ~~\textup{if $p>3$} \\
\frac{8}{\Gamma_p(1/4)^2} & ~~\textup{if $p=3$},
\end{cases}
\end{multline*}

where $h(-p)$ is the class number of the imaginary quadratic field $\Q(\sqrt{-p})$.

\end{thm}
\begin{rem}
It is likely that, for arbitrary odd $M \geq 1$  coprime to $p$
\[
\gamma =  \left( \frac{2}{p} \right) \lim_{\itop{m \rightarrow \infty}{m \equiv (M-1)/2 \pmod 2}} \frac{5C((Mp^{m+1}+1)/4)}{3C((Mp^m -1)/4)}.
\]
While we do not prove that, we indicate clearly where we  reduce the consideration to $M=1$ in the course of the proof of the theorem.
\end{rem}

\begin{rem}

General results of Ogus \cite{Ogus} (see also \cite[Theorems 4.6.3,5.3.9]{Andre}) seemingly predict the existence of formulas involving some special values of $p$-adic $\Gamma$-function as in Theorem \ref{thm_mu_32}. 
The current author, however, is unable to find a way to specialize these results to Theorem \ref{thm_mu_32}. 
We provide an independent proof of Theorem \ref{thm_mu_32} in this paper (see Remark \ref{rem_gamma_fraction} for a discussion of the parallelism with the well-known complex-analytic analog of this result).

\end{rem}

Ideas involved in the proof of Theorem \ref{thm_gamma_modular} are discussed and commented in Section \ref{sec_overview}. 
In particular, Theorem \ref{thm_mu_Eg} which implies Theorem  \ref{thm_gamma_modular} is stated in Section \ref{sec_overview}. 
In Section \ref{sec_example}, we concentrate on the $N=32$ case and prove Theorem \ref{thm_mu_32}. The proof illustrates some of the ideas employed in this paper. 
In Section \ref{sec_pf_thm}, we prove Theorem \ref{thm_mu_Eg}; that accomplishes the proof of Theorem \ref{thm_gamma_modular}.

An alternative way to prove Theorem \ref{thm_mu_Eg} is briefly laid out in Remark \ref{rem_infty} in Section \ref{sec_overview}  below. 
Roughly, one can make use of \cite[Corollary 3.8]{BKY} which, in its turn, is proved using a $p$-adic analog of Legendre's period relation presented in \cite[Theorem 3.7(i)]{BKY}.
In this way, the calculation of the key quantity $\mu_p$ in Theorem \ref{thm_mu_Eg} (and proving $\mu_p \neq 0$) involves working in an ambient ring $\bf{A}_{\text{cris}}$ (the $p$-adic periods). 
This project was inspired by the idea that, since $\mu_p \in \Q_p$ by construction, there should be a more straightforward way to calculate it (and prove $\mu_p \neq 0$),
bypassing the Colmez integration theory, therefore without embarking to the ambient ring of $p$-adic periods.

\section*{Acknowledgements}

The author expresses his gratitude to the reviewer whose comments allowed the author to substantially improve the presentation.

\section{Overview of the  proof} \label{sec_overview}

In this section, we preserve all notations and conventions introduced in Section \ref{sec_intro}. 

We state Theorem \ref{thm_mu_Eg} which implies  Theorem \ref{thm_gamma_modular}, and 
state congruences \ref{equ_mu_E} and \ref{equ_E_unit} which imply Theorem \ref{thm_mu_Eg}.  
The proof of these congruences (and therefore of Theorem \ref{thm_mu_Eg}) is postponed to Section \ref{sec_pf_thm}.


Theorem \ref{thm_mu_Eg}. is a result on the formal group law associated with a model of the elliptic curve related to $g$ via the Eichler - Shimura theory. 

Recall that (see e.g. \cite{Zagier} for details) the errors of modularity of the function $\cE_g$ form a lattice $\Lambda \subset \C$, and we have a covering map 
$X_0(N) = \overline { \Gamma_0(N) \backslash \mathfrak{H}} \rightarrow \C/\Lambda =E$ from the modular curve to a strong Weil elliptic curve $E$ defined by $\tau \mapsto \cE_g(\tau)$. Thus the differential $dz$ on $E$ pulls back to the differential $-2\pi i g(\tau) d\tau$ on $X_0(N)$. Furthermore, $E$ admits a rational model
\begin{equation} \label{equ_ell_curve_model}
y^2=4x^3-g_2x-g_3 
\end{equation}
with
\[
 g_2(\Lambda) = 60\sum_{\itop{m \in \Lambda}{m \neq 0}} m^{-4}, ~~ \text{and} ~~ g_3(\Lambda)= 140 \sum_{\itop{m \in \Lambda}{m \neq 0}} m^{-6}.
\]
The model \eqref{equ_ell_curve_model} is defined over $\Z_p$ for all but possibly finitely many primes $p$. From now on, we discard the finitely many possible exceptions and consider only the primes $p$ such that $g_2,g_3 \in \Z_p$.

The Weierstrass $\zeta$-function associated with $\Lambda$
\[
\zeta(\Lambda,z) = \frac{1}{z} + \sum_{\itop{m \in \Lambda}{m \neq 0}} \left( \frac{1}{z-m} + \frac{z}{m^2} + \frac{1}{m} \right)
\]
is a meromorphic function on $\C$, which is not $\Lambda$-periodic. However, by a classical observation which goes back to Eisenstein 
one can make this function $\Lambda$-periodic by adding to it a linear combination of $z$ and $\overline{z}$. 
More precisely,  there exist $\lambda_\infty,\mu_\infty \in \C$ such that the function
\[
R(z):= \zeta(\Lambda,z) - \lambda_\infty z
 - \mu_\infty \overline{z}
\]
is $\Lambda$-periodic. As such, $R(z)$ lifts to a function  on  the upper half-plane 
\[
F(\tau):=R(\cE_g(\tau))
\]
which is modular on $\Gamma_0(N)$, while not analytic. 

In this paper, we concentrate on a $p$-adic analog of quantity $\mu_\infty$. 
A classical calculation which can be traced back to Eisenstein (see \cite{Weil}) implies
\begin{equation} \label{equ_muinf}
\mu_\infty = \frac{\pi}{a(\Lambda)},
\end{equation}
where $a(\Lambda)$ is the volume of the fundamental parallelogram of the lattice $\Lambda$, and we conclude that $\mu_\infty$ is a positive real number, in particular, $\mu_\infty \neq 0$.

We split $F=F^++F^-$ the function into its meromorphic and antiholomorphic parts with
\[
F^+(\tau):=\zeta(\Lambda,\cE_g(\tau)) - \lambda_\infty\cE_g(\tau) ~~ \textup{and} ~~ F^-(\tau)=\mu_\infty \overline{\cE_g(\tau)},
\]
and observe a close similarity between the function $F$ and the  harmonic Maass form $f$.  Specifically, it follows from \cite[Proposition 2]{adele} that there exists a non-zero integer $C \in \Z^*$ such that 
\begin{equation} \label{equ_F+f+Q}
F^+ - Cf^+ \in \Q\llparenthesis q \rrparenthesis,
\end{equation}
and
for all but possibly finitely many primes $p$
\begin{equation} \label{equ_F+f+}
F^+ - Cf^+ \in \Z_p\llparenthesis q \rrparenthesis.
\end{equation}
We firstly combine \eqref{equ_F+f+Q} with  \eqref{equ_def_alpha} to conclude that, since obviously $\zeta(\Lambda, \cE_g) \in \Q \llparenthesis q \rrparenthesis$ (see \eqref{eq_wei_p} below),  
\begin{equation} \label{equ_choice_of_alpha}
\alpha = -\frac{\lambda_\infty}{C}
\end{equation}
is a legitimate choice of $\alpha$. We assume this choice, rewrite \eqref{equ_F+f+}  (for arbitrary $\beta,\gamma \in \Q_p$) as
\begin{equation} \label {equ_relate_to_good}
F^+ + \lambda_\infty\cE_g - C\beta \cE_g - C\gamma \frac{1}{p}\cE_g|V - C(f^+ - \alpha\cE_g - \beta \cE_g-\gamma \frac{1}{p}\cE_g|V ) \in \Z_p\llparenthesis q \rrparenthesis
\end{equation}
and use \cite[Proposition 5]{adele} (exactly as in the proof of Proposition \ref{prop_coupling}, independently on whether $b(p)=0$ as remarked after the proof) 
to pick $\beta, \gamma \in \Q_p$ such that \eqref{equ_except_gamma} holds. That happens if and only if, with these same $\beta, \gamma \in \Q_p$, 
\[
\zeta(\Lambda, \cE_g) - C\beta \cE_g - C\gamma \frac{1}{p}\cE_g|V \in \Z_p\llparenthesis q \rrparenthesis.
\]
We now easily derive Theorem \ref{thm_gamma_modular} from the following statement which will be proved in Section \ref{sec_pf_thm} 
(specifically, the exceptional $\gamma$ in \eqref{equ_except_gamma} coincides with $\mu_p/C$ in the following theorem).

\begin{thm} \label{thm_mu_Eg}

Let $E$ be an elliptic curve  with the rational model \textup{\eqref{equ_ell_curve_model}} as above and let $p$ be a prime of good reduction.

Let $\lambda_p,\mu_p \in \Q_p$ be such that 
\begin{equation} \label{equ_lambda_mu_Eg}
\ord_p \left(  \zeta (\Lambda, \cE_g) - \lambda_p \cE_g - \mu_p \frac{1}{p} \cE_g |V \right ) > -\infty.
\end{equation}


Then $\mu_p=0$ if and only if the reduction at $p$ is ordinary.

\end{thm}

Recall that, for a prime $p$ of good reduction, the reduction of $E$ at $p$ is ordinary whenever $b(p) \neq 0$ and supersingular whenever $b(p)=0$.

\begin{rem}

The uniqueness of $\lambda_p,\mu_p \in \Q_p$ in Theorem \ref{thm_mu_Eg} (assuming their existence) is immediate. 
Their existence for almost all primes $p$ of good reduction follows form \eqref{equ_except_gamma}  and \eqref{equ_relate_to_good}. 
It also follows from  \cite[Theorem 1.1(1)]{BGK}) that $\mu_p=0$ whenever the reduction of $E$ at $p$ is ordinary. 
The  new statement which we prove in this paper is $\mu_p \neq 0$  whenever the reduction of $E$ at $p$ is supersingular (equivalently,  $b(p)=0$).

\end{rem}

In order to prove Theorem \ref{thm_mu_Eg}, we consider it as a statement about the {\sl formal group law} (FGL) associated with the elliptic curve $E$. 
The seminal result of Honda \cite{Honda} (see also \cite{Hill}) claims a strict isomorphism over $\Z$ between this FGL and the FGL over $\Z$ whose logarithm is given by the series $\cE_g \in \Q\llbracket q \rrbracket$. As the chosen model \eqref{equ_ell_curve_model} may be not minimal, the isomorphism holds over $\Z_p$ for all but possibly finitely many primes $p$. 
It is proved in \cite[Corollary 1(a)]{Guerzhoy_P} that the quantities $\lambda_p$ and $\mu_p$ are invariant under a strict FGL isomorphism. 
We thus make a convenient choice of a FGL isomorphic to that associated with $E$ over $\Z_p$, for which the $\mu_p \neq 0$ statement becomes apparent.

To be more specific we need to introduce some standard notations. Recall that values on $\Lambda$ of the Eisenstein series are the Laurent coefficients of the Weierstrass $\wp$-function:
\begin{equation} \label{eq_wei_p}
\wp(z, \Lambda) = \frac{1}{z^2} + 2\sum_{n \geq 1} G_{2n+2} \frac{z^{2n}}{(2n)!}, ~~ ~~ \textup{where} ~~ ~~ G_{2k} = \frac{(2k-1)!}{2} \sum_{\itop{m \in \Lambda}{m \neq 0}} 
\frac{1}{m^{2k}}.
\end{equation}
In particular, $G_4 = (1/20)g_2$ and $G_6=(3/7)g_3$, and it is immediate from the recursion formula that $G_{2k} \in \Q$ for $k \geq 1$.
Another standard normalization for the Eisenstein series is
\[
E_{2k} = -\frac{4k}{B_{2k}} G_{2k}.
\]
Theorem \ref{thm_mu_Eg} follows immediately from the following two statements which will be proved in Section \ref{sec_pf_thm}:

\begin{equation} \label{equ_mu_E}
\mu_p \equiv- \frac{1}{12}E_{p+1} \pmod p
\end{equation}

and

\begin{equation} \label{equ_E_unit}
E_{p+1} \not\equiv 0 \pmod p ~~ \text{for $p>3$ and } ~~ E_4/12   \not\equiv 0 \pmod 3
\end{equation}
as soon as \eqref{equ_ell_curve_model} has good supersingular reduction at $p$.
(We will see that $E_4 \equiv 0 \pmod 3$ for $p=3$.)

\begin{rem}
The statements \eqref{equ_mu_E} and \eqref{equ_E_unit} are already about an elliptic curve (which may not necessarily come from the Eichler - Shimura construction). 
Our argument in  Section \ref{sec_pf_thm} generalizes {\sl mutatis mutandis} to an elliptic curve with a model \eqref{equ_ell_curve_model} defined over the ring of integers $\cO_K$ in an algebraic number field $K$ and a prime ideal $\gp$ where  \eqref{equ_ell_curve_model} has good supersingular reduction. 
We do not pursue this generality and stick to a rational elliptic curves $E$ in this paper.
\end{rem}

\begin{rem} \label{rem_infty}

Note that $d\zeta$ is a differential of the second kind on the elliptic curve $E$ and let $\omega=dz$ be the (unique up to a non-zero constant multiple) holomorphic differential of the first kind.

The observation that $R(z)$ is $\Lambda$-periodic becomes (cf. \cite[Section 1]{Katz_eis})
\begin{equation} \label{equ_DR}
d \zeta = \lambda_\infty \omega + \mu_\infty \overline{\omega} ~~ \text{in} ~~ H^1(E, \C).
\end{equation}
Similarly, \eqref{equ_lambda_mu_Eg} can be slightly strengthened and translated to an identity in $H^1_{\textup{cris}}(\hat{E})$ of the FGL $\hat{E}$ over $\Z$ with logarithm $\cE_g \in \Q \llparenthesis q \rrparenthesis$ (considered as a formal power series in $q$)
\begin{equation} \label{equ_cris}
d\hat{\zeta} +\frac{dq}{q^2} = \lambda_p \hat{\omega} + \mu_p {\hat{\omega}^*} ~~ \text{in} ~~ H^1_{\textup{cris}}(\hat{E}).
\end{equation}
While we do not use it in this paper, an interested reader may consult e.g. \cite[Section 2]{BKY} (and the literature cited in loc cit.) 
for the definition of $H^1_{\textup{cris}}(\hat{E})$ and detailed discussion of the FGL crystalline cohomology. 

Here 
\[
\hat{\omega} = d\cE_g = \left( \sum_{n\geq 1} b(n) q^{n-1} \right)  dq
\]
is the invariant differential on $\hat{E}$, and 
\[
\hat{\omega}^* = d\frac{1}{p} \left( \cE_g|V \right) = \left( \sum_{n\geq 1} b(n) q^{pn-1} \right)  dq.
\]
Since $E$ has good supersingular reduction at $p$ (therefore $b(p)=0$ and $b(p^2)=-p$), the height of the modulo $p$ reduction of $\hat{E}$ is $2$, and the existence of $\lambda_p, \mu_p \in \Z_p$ follows from 
\cite[Lemma 3.4, Proposition 2.3]{BKY}. 

The parallelism between \eqref{equ_DR} and \eqref{equ_cris} is well-known (see, e.g., a discussion after \cite[Definition 4.2]{BKY}).

Note that while  $\lambda_\infty=0$ or $\lambda_p=0$ may well happen (both  indeed occur in our example), the quantity 
\[
\mu_\infty = \pi/a(\Lambda) \neq 0 
\]
is manifestly non-zero, and we prove the parallel statement on $\mu_p$. 
The quantity $\pi$ above comes from the Legendre period relation (see e.g. \cite[1.2.5,1.3.3]{Katz_eis}). A $p$-adic analog of the Legendre period relation is presented in \cite[Theorem 3.7(i)]{BKY}. The fact that $\mu_p \neq 0$ is closely related to and probably can be derived from \cite[Corollary 3.8]{BKY} 
(see also Remark 3.10 in loc.cit. for an indication of an alternative method). In this paper, we employ a different method which avoids the use of a much larger $p$-adic periods ring $\bf{A}_{\textup{cris}}$ (and the Colmez integration which leads to the consideration of this ring). 

\begin{rem}
It may be illuminating to note that $\mu_\infty = \pi/a(\Lambda)$ is, by \cite[1.1.6]{Katz_eis}, the value at $(E,dz)$ of a $\mathcal{C}^\infty$-modular form of weight $(0,1)$ on $SL_2(\Z)$ over $\C$. 
\end{rem}

\end{rem}

\section{Example} \label{sec_example}

In this section, we illustrate our results with the  special case 
\[
g(\tau)=\eta^2(4\tau)\eta^2(8\tau) \in S_2(32).
\]

All specific calculations are done using the PARI/GP package \cite{PARI2}.
This example has been considered in \cite{EGO}, where the elliptic curve (called {\sf 32A} in the standard classification \cite{BK}) with a model $y^2=x^3-x$ is mentioned. 
However, the period lattice 
\[
\Lambda= \Omega \left\langle \frac{1-i}{2} ,1 \right\rangle 
\]
of the errors of modularity of $\cE_g$ determine  the strong Weil curve $E=\C/\Lambda$ with a (manifestly non-minimal) model 
\begin{equation} \label{eq_ell_32}
y^2=4x^3+16x
\end{equation}
(called  {\sf 32B} in \cite{BK}). The curves {\sf 32A} and {\sf 32B} are isogeneous while not isomorphic over $\Q$.
A well-known calculation  (see e.g. \cite[Proposition 6.2]{Husemoller} implies that
\[
\Omega = \frac{\sqrt {\pi}}{2} \frac{\Gamma(1/4)}{\Gamma(3/4)}.
\]
The covering $\Gamma_0(32) \backslash \mathfrak{H} \rightarrow E$ becomes an isomorphism, therefore the possibility of finitely many exceptions in Theorem \ref{thm_gamma_modular} does not materialize in this example. 

In \cite{EGO},  a  harmonic Maass form which is good for $g$ is constructed using a Hauptmodule for $\Gamma_0(16)$
\[
L(\tau) = \frac{ \eta^6(8\tau)}{\eta^2(4\tau)\eta^4(16\tau)}
\]
which, in its turn, is explained using hypergeometric functions. 
Specifically, (in our notations, with $\alpha=0$)
\[
Df^+=D\cF_\alpha := -g(\tau)L(2\tau) \in S_2^!(32).
\]

Another way to look at the modular function $L(2\tau)$ is to observe that it coincides with the pullback of the Weierstrass $\wp$-function from $E$ to $X_0(32)$:
\[
\wp(\Lambda,\cE_g(\tau)) = L(2\tau).
\]
(As it is an equality of two modular functions in $M_0^!(32)$, it can be easily proved by calculating sufficiently many $q$-expansion coefficients.)
It now follows from $\zeta(z) = -\int \wp(z)dz$ that 
\begin{multline*}
\zeta(\Lambda, \cE_g(\tau)) = -\int \wp(\Lambda,\cE_g(\tau)) d(\cE_g(\tau)) = \\ -\int \wp(\Lambda, \cE_g(\tau)) g(\tau) \frac{dq}{q}  =- \int   g(\tau) L(2\tau) \frac{dq}{q}.
\end{multline*}
Equation \eqref{equ_F+f+} now specializes to $F^+-f^+ = 0$ (with $\lambda_\infty=0$ and $C=1$). In other words,  $R(\cE_g(\tau))=\zeta(\Lambda, \cE_g(\tau) )- \mu_\infty \overline {\cE_g(\tau)}$ in this example {\sl coincides} with the  harmonic Maass form constructed in \cite{EGO}.

\begin{rem}
More congruences similar to those obtained in \cite{EGO} can be obtained  if one starts with  another weight two meromorphic modular form instead of $-g(\tau)L(2\tau)$. For instance, the function $E_4(4\tau)/g(\tau)$ with usual
\[
E_4(\tau) = 1+240 \sum_{n \geq 1} \left( \sum_{d|n}d^3 \right)q^n
\]
does the job. Indeed, it is not difficult to establish the identity of modular functions
\begin{multline} \label{equ_20_zeta_int}
20\zeta(E_g(\tau)) + \int \frac{E_4(4\tau)}{g(\tau)} \frac{dq}{q} \\ = \frac{56P(4\tau)-32P(8\tau)+160P(16\tau)-640P(32\tau)}{g(\tau)},
\end{multline}
where 
$
P(\tau)=1/24-\sum_{n \geq 1} \left( \sum_{d|n}d \right)q^n.
$
Note that the right side of \eqref{equ_20_zeta_int} (being multiplied by $3$) has integral $q$-expansion coefficients.
We now consider the $q$-expansion 
$
E_4(4\tau)/g(\tau) = \sum_{n \geq -1} B(n) q^n,
$
and conclude that, for $p \equiv 3 \pmod 4$,
\[
\lim_{m \rightarrow \infty} \frac{\left. \left( E(4\tau)/g(\tau) \right) \right| U^{2m+1}}{B(p^{2m+1})} = g.
\]
The calculations sketched in this remark illustrate our way to derive Theorem \ref{thm_gamma_modular} from Theorem \ref{thm_mu_Eg} using \eqref{equ_relate_to_good}.
\end{rem}

The proof of Theorem \ref{thm_mu_32} below illustrates some ideas involved in the proof of Theorem  \ref{thm_gamma_modular} .

\begin{proof}[Proof of Theorem \ref{thm_mu_32}]

It is easy to see that the model \eqref{eq_ell_32} of the elliptic curve $E$ can be parametrized as
\[
x=\frac{1}{2t^2}\left(1+\sqrt{1-16t^4} \right), ~~ y=-\frac{1}{t^3}\left(1+\sqrt{1-16t^4} \right).
\]
(Formulas of this kind for the lemniscate elliptic curve likely date back to 18th century. 
We refer to \cite{Yasuda} for a vast generalization, and note that our specific parametrization can be obtained from \cite[Corollary 8a]{Yasuda}.) 

The FGL $\hat{E}$ associated with the model $E\ : \ y^2=4x^3 + 16x$ is determined by its logarithm
\[
l(t)=\int \frac{dx}{y} = \int \frac{dt}{\sqrt{1-16t^4}} = \sum_{n \geq 0} \frac{4^n}{4n+1} \binom{2n}{n} t^{4n+1} \in \Q \llp t \rrp.
\]
Honda's theorem \cite{Honda} implies that this FGL is strictly isomorphic over $\Z_p$ to the FGL with logarithm $\cE_g$ for all primes $p>2$.

\begin{rem}
The restriction $p>2$ (in general, we have to exclude finitely many primes) comes from the consideration of a non-minimal model. 
This restriction is non-essential in this case since we anyway consider only supersingular primes $p \equiv 3 \pmod 4$ now. 
A separate instance when a finite set of primes may be excluded from the consideration is related to \eqref{equ_F+f+}; that does not happen now, because $X_0(32)$ has genus $1$ and the modular parametrization is an isomorphism.
\end{rem}

In specific terms, Honda's theorem claims that, for $p\geq 3$, the formal power series 
\[
t(q):=l^{-1}(\cE_g(q)) \in \Z_p\llbracket q \rrbracket
\]
has $p$-integral coefficients, while  a priori, of course, merely $t(q) \in \Q\llbracket q \rrbracket$.


Since $E$ has a supesingular reduction at primes $p \equiv 3 \pmod 4$, the modulo $p$ reduction of $\hat{E}$ has height $2$. 
Let 
\begin{multline*}
\xi(t)  =\\  \int x  \frac{dx}{y} = \int \frac{1}{2t^2} \left( \frac{1}{\sqrt{1-16t^4}} + 1\right) dt = - \frac{1}{t} + \frac{1}{2} \sum_{n \geq 1}\frac{4^n}{4n-1} \binom{2n}{n} t^{4n-1},
\end{multline*}
so that $\xi(t(q))=-\zeta(\cE_g(q))$.

It follows from the addition law for the Weierstrass $\zeta$-function and \cite[Theorem 5.3.3]{Katz_crys} (see also \cite{adele}, \cite[Proposition 2.3(iii), Lemma 3.4]{BKY} for a more elementary and detailed consideration) that there exist $\lambda_p,\mu_p \in \Z_p$ such that 
\[
\xi(t) + \lambda_p l(t) + \frac{1}{p} \mu_p l(t^p) \in \Z_p\llbracket t \rrbracket
\]

It follows from  \cite[Corollary 1(a)]{Guerzhoy_P} that the quantities $\lambda_p$ and $\mu_p$ are invariant with respect to a strict FGL isomorphism, therefore they coincide with 
these quantities introduced in \eqref{equ_lambda_mu_Eg}, and there is no apparent abuse of notations here.

\begin{rem}

These are two different ways to prove the existence of $\lambda_p$ and $\mu_p$: the  application of an argument along the lines of \cite[Theorem 1.1]{BGK} which is applicable whenever $g$ is a cusp Hecke eigenform of arbitrary even weight and the argument just sketched above which is applicable to an elliptic curve over a number field. 
In the setting under consideration in this paper, both are applicable, and we employ the interplay between them. 
In particular, the explicit calculation of $\mu_p$ becomes available with the specific parametrization of the elliptic curve considered here.

\end{rem}

We now want to calculate $\mu_p$ (that coincides with $\gamma$ in Theorem \ref{thm_mu_32})   such that 

\begin{equation} \label{equ_mu_32}
\xi(t) +  \frac{1}{p} \mu_p l(t^p) \in \Z_p\llbracket t \rrbracket.
\end{equation}
The denominator of the coefficient of $t^{4n-1}$ in  $\xi$ is divisible by $p$ only if $4n-1 \equiv 0 \pmod p$, and letting $4n-1 = p(4s+1)$ we obtain the condition
\begin{equation} \label{equ_cond_binom}
\frac{1}{2}4^n \binom{2n}{n} \frac{1}{p(4s+1)} + \mu_p \frac{1}{p} 4^s \frac{1}{4s+1} \binom{2s}{s}  \in \Z_p.
\end{equation}
Let 
\[
4s+1 = p^mM ~~ \textup{with $m \geq 0$ and $p \nmid M$ }
\]
Obviously, $M$ is odd, $m \equiv (M-1)/2 \pmod 4$, and we conclude that
\[
4^{n-s}=4^{(p^{m+1}M-p^mM+1)/4} = 2 \left( 2^{(p-1)/2} \right)^ {p^mM} \equiv 2 \left(\frac{2}{p} \right) \pmod {p^{m+1}}
\]
since $M$ is  odd and 
\[
2^{(p-1)/2} \equiv  \left(\frac{2}{p} \right) \pmod p.
\]
Thus \eqref{equ_cond_binom} reduces to
\[
\left(\frac{2}{p} \right) \binom{(p^{m+1}M+1)/2}{(p^{m+1}M+1)/4} \equiv -\mu_p \binom{(p^{m}M-1)/2}{(p^{m}M-1)/4}  \pmod {p^m}
\]
Since we know in advance that $\mu_p$ in \eqref{equ_mu_32} exists and is independent on $M$, we may carry the calculation only for $M=1$ (and even $m$):
\[
\left(\frac{2}{p} \right) \binom{(p^{m+1}+1)/2}{(p^{m+1}+1)/4} \equiv -\mu_p \binom{(p^{m}-1)/2}{(p^{m}-1)/4}  \pmod {p^m}.
\]
We now use the identity
\[
\ord_p \binom{(p^{m+1}+1)/2}{(p^{m+1}+1)/4} = \ord_p\binom{(p^{m}-1)/2}{(p^{m}-1)/4} = \frac{m}{2} <m.
\]
One can prove this identity with a straightforward argument using the well-known fact  that $\ord_p(j!) = (j-S_j)/(p-1)$, where $S_j$ is the sum of digits of the integer $j$ written in base $p$.
(It is likely that a similar fact is true with an arbitrary odd $M>1$; the author has not verified that.) We thus have that
\[
\mu_p =  -\left(\frac{2}{p} \right) \lim_{\itop{m \rightarrow \infty}{ 2|m}} 
\left.  \binom{(p^{m+1}+1)/2}{(p^{m+1}+1)/4} \right/ \binom{(p^{m}-1)/2}{(p^{m}-1)/4}.
\]
We now recall a definition of the Catalan numbers
\begin{equation} \label{eq_catalan}
C(n) = \frac{1}{n+1} \binom{2n}{n}
\end{equation}
to finish our proof of the first claim of  Theorem \ref{thm_mu_32}.

Abbreviate $a=(p^m -1)/4 \in \Z \hookrightarrow \Z_p$ and note that the limit $m \rightarrow \infty$ corresponds to $a \rightarrow -1/4$.  Rewrite $\mu_p$ as
\begin{multline} \label{eq_mu_factors}
\mu_p =  -\left(\frac{2}{p} \right) \lim_{a \rightarrow -1/4} 
\left. \binom{2pa+ (p+1)/2}{pa+(p+1)/4} \right/ \binom{2a}{a} \\
= 
- \left(\frac{2}{p} \right) \lim_{a \rightarrow -1/4} 
 \left. \binom{2pa}{pa} \right/ \binom{2a}{a} 
\frac{(2pa+1) \ldots (2pa+ (p+1)/2)} {(pa+1)^2 \ldots (pa+(p+1)/4)^2}.
\end{multline}

We calculate the limits of the factors in \eqref{eq_mu_factors} separately.

Recall that Morita's $p$-adic $\Gamma$-function extends continuously to $\Z_p$ the function defined by its values on integers $x \geq 2$ by
\[
\Gamma_p(x) = (-1)^x\prod_{\itop{1 \leq j < x}{p \nmid j}} j.
\]
An observation credited to L. van Hamme in \cite[Section 7.1.6]{Robert} allows us to conclude that
\[
 \left. \binom{2pa}{pa} \right/ \binom{2a}{a}  = \frac{\Gamma_p(2pa)}{\Gamma_p(pa)^2} \ \xrightarrow[a \rightarrow -1/4]{}  \  \frac{\Gamma_p(-p/2)}{\Gamma_p(-p/4)^2}.
\]
Since none of the integers $1,2, \ldots ,(p+1)/2$ is divisible by $p$, and $(p+1)/2$ is even,
\begin{multline*}
(2pa+1) \ldots (2pa+ (p+1)/2) \\ = \frac{\Gamma_p(2pa+1+(p+1)/2)}{\Gamma_p(2pa+1)}  \ \xrightarrow[a \rightarrow -1/4]{}  \  \frac{\Gamma_p(3/2)}{\Gamma_p(1-p/2)}.
\end{multline*}

Similarly, 
\begin{multline*}
(pa+1)^2 \ldots (pa+ (p+1)/4)^2 \\ = \frac{\Gamma_p(pa+1+(p+1)/4)^2}{\Gamma_p(pa+1)^2}  \ \xrightarrow[a \rightarrow -1/4]{}  \  \frac{\Gamma_p(5/4)^2}{\Gamma_p(1-p/4)^2}.
\end{multline*}
Collecting the three above observations together, we obtain that
\[
\mu_p =  -\left(\frac{2}{p} \right) \frac{\Gamma_p(-p/2)}{\Gamma_p(-p/4)^2} \frac{\Gamma_p(3/2)}{\Gamma_p(1-p/2)} \frac{\Gamma_p(1-p/4)^2}{\Gamma_p(5/4)^2}.
\]
Furthermore, since $\Gamma_p(1+x)=-\Gamma_p(x)$ for $x \in p\Z_p$ and $\Gamma_p(1+x)=-x\Gamma_p(x)$ for $x \in \Z_p^*$, we conclude the proof of the second equality in Theorem \ref{thm_mu_32}:
\[
\mu_p = -8 \left(\frac{2}{p} \right)  \frac{\Gamma_p(1/2)}{\Gamma_p(1/4)^2}.
\]

In order to prove the last equality in Theorem \ref{thm_mu_32}, we need to show that, for $p \equiv 3 \pmod 4$,
\[
\Gamma_p(1/2) = 
\begin{cases} 
(-1)^{(1+h(-p))/2} & \textup{if $p>3$} \\
1 & \textup{if $p=3$.} 
\end{cases}
\]

It immediately follows from the reflection formula for $\Gamma_p$ (see \cite[Theorem 7.1.2(5)]{Robert}) that $\Gamma_p(1/2)=\pm 1$ for $p \equiv 3 \pmod 4$ and it is easy to see that
\[
\Gamma_p(1/2) \equiv \Gamma_p((p+1)/2) = (-1)^{(p-1)/2} ((p-1)/2)! =  ((p-1)/2)! \pmod p.
\]
The congruence
\[
 ((p-1)/2)! \equiv (-1)^{(1+h(-p))/2} \pmod p
\]
for $p\equiv 3 \pmod 4$ and $p>3$ is established in \cite{Mordell}.

\end{proof}

\begin{rem} \label{rem_gamma_fraction}

In view of the parallelism between $\mu_p$ and $\mu_\infty$ observed in \eqref{equ_cris} and \eqref{equ_DR}, it is interesting to juxtapose their explicitly calculated values in this example. 

We substitute the lattice $\Lambda$ from this example into \eqref{equ_muinf} to obtain (using the Legendre relation $\Gamma(x)\Gamma(1-x) = \pi/\sin(\pi x)$) that 
\[
\mu_\infty = \frac{\pi}{a(\Lambda)} = 8 \frac{\Gamma(3/4)^2}{\Gamma(1/4)^2} = 16 \frac{\pi^2}{\Gamma(1/4)^4} = 16 \frac{\Gamma(1/2)^4}{\Gamma(1/4)^4}.
\]
That compares well with
\[
\mu_p^2 = 64 \frac{\Gamma_p(1/2)^2}{\Gamma_p(1/4)^4} =  64 \frac{\Gamma_p(1/2)^4}{\Gamma_p(1/4)^4}
\]
(since $\Gamma_p(1/2)^2 = (-1)^{(p+1)/2} = 1$ for $p \equiv 3 \pmod 4$; see also a comment in \cite[Section 7.1.2]{Robert} on the parallelism between $\Gamma(1/2)^2=\pi$ and $\Gamma_p(1/2)^2$).

While the similarity between these formulas may be expected from the prospective of \cite{Gross}, the author cannot find a conceptual explanation for 
the appearance of $\mu_p^2$ instead of $\mu_p$ itself.

Note also that, while $\mu_p \in \Q_p$ by construction, $\mu_\infty$ is well-known (see e.g. \cite{Waldschmidt}) to be transcendental (over $\Q$).

\end{rem}

\section{Proof of Theorem \ref{thm_mu_Eg}} \label{sec_pf_thm}

In  Section  \ref{sec_overview}, the proof of Theorem   \ref{thm_mu_Eg}  is reduced to the congruences \eqref{equ_mu_E} and \eqref{equ_E_unit} for the values of the Eisenstein series $E_{p+1}$. We prove  \eqref{equ_mu_E} and \eqref{equ_E_unit} in this Section.

For an elliptic curve over $\Q$ with a model \eqref{equ_ell_curve_model}, we denote by $\cG$ the associated FGL (with its differential $\omega=dx/y$ and the parameter $t=-2x/y$, see e.g. \cite[Section IV.1]{Silverman} for details of the construction). The logarithm of $\cG$ is the formal power series (with rational coefficients and no constant term) 
\[
\ell(t) = \int \omega \in t\Q \llbracket t \rrbracket.
\]

Let now $p$ be a prime such that the model \eqref{equ_ell_curve_model} has good supersingular reduction at $p$ (that happens for all but possibly finitely many  primes 
of good supersingular reduction of $E$ 
since the model \eqref{equ_ell_curve_model} may be and typically is not minimal). The reduction of $\cG=\cG(u,v) \in \Z_p\llbracket u,v \rrbracket$ at $p$ has height $2$, and, by \cite[Theorem 5.3.3]{Katz_crys}, the Dieudonn\'e module $\D(\cG)$ is a free $\Z_p$-module of rank $2$. Recall that (specializing to our notations)
\[
\D(\cG):= Z(\cG)/B(\cG),
\]
where 
\[
Z(\cG):= \left\{ \xi \in t\Q_p \llbracket t \rrbracket  \Big\vert \  \xi' \in \Z_p \llbracket t \rrbracket \hspace{1mm} \text{and} \hspace{1mm} \xi(\cG(u,v)) - \xi(u) - \xi(v) \in \Z_p\llbracket u,v \rrbracket \right\}
\]
and
\[
B(\cG):= \left\{ \xi \in t\Z_p \llbracket t \rrbracket \right\}.
\]
It is known (see e.g. \cite[Proposition 2.3(iii)]{BKY} or \cite{adele}) that the two formal power series $\ell(t),(1/p)\ell(t^p) \in \Q \llbracket t \rrbracket$ form a basis of $\D(\cG)$.

The addition law for Weierstrass $\zeta$-function 
\[
\zeta(\Lambda, \ell(t)) = -\int x \omega = -\int x \frac{dx}{y} \in \frac{1}{t} + t\Q  \llbracket t \rrbracket 
\]
implies (see \cite[Lemma 2.4]{BKY} or \cite[Proposition 12]{adele}) that $\zeta(t) - 1/t \in \D(\cG)$, therefore there exist $A,B \in \Z_p$ such that
\[
\zeta(\Lambda, \ell(t))  - A \ell(t) - B\frac{1}{p} \ell(t^p) \in \frac{1}{t} +t\Z_p \llbracket t \rrbracket.
\]
It follows from \cite[Corollary 1(a)]{Guerzhoy_P} that the quantities $A$ and $B$ are invariant under a strict FGL isomorphism, specifically, for any power series 
\[
t \in u+ u^2\Z_p\llbracket u \rrbracket,
\]
\[
\zeta(\Lambda, \ell(t))  - A \ell(t) - B\frac{1}{p} \ell(t^p) \in \frac{1}{u} +u\Z_p \llbracket u\rrbracket.
\]
In particular, since the theorem of Honda \cite{Honda} implies an isomorphism (over $\Z_p$ for all  but possibly finitely many primes $p$) between $\cG$ and the FGL determined by it logarithm $\cE_g \in \Q \llbracket q \rrbracket$, we conclude that 
\[
A = \lambda_p ~~ \textup{and} ~~ B= \mu_p.
\]

\begin{rem}
Observe that a statement stronger than \eqref{equ_lambda_mu_Eg}
\[
 \zeta (\Lambda, \cE_g) - \lambda_p \cE_g - \mu_p \frac{1}{p} \cE_g |V  \in \frac{1}{q}+ \Z_p\llp q \rrp
\]
is actually true. This provides for an alternative method to prove (a refinement of) \cite[Theorem 1.1]{BGK}), although only in the case when the shadow $g$ has weight $2$ and rational, not merely algebraic, integral coefficients.

\end{rem}

In order to prove Theorem \ref{thm_mu_Eg}, we again employ the freedom to pick any FGL which is strictly isomorphic over $\Z_p$ to $\cG$. Specifically, 
observe that, since $p$ is a prime of good reduction and $b(p)=0$, the formal group laws with logarithms $\cE_g$ and 
\[
l_t(u) = \sum_{n \geq 0} (-1)^n\frac{u^{p^{2n}}}{p^n}
\]
are strictly isomorphic over $\Z_p$. (The latter is the $p$-typical FGL whose existence is established in \cite[Theorem 15.2.9]{Hazewinkel}.) It follows that
\[
 \zeta (\Lambda, l_t(u)) - \lambda_p l_t(u) - \mu_p \frac{1}{p} l_t(u^p)  \in \frac{1}{u}+ \Z_p\llp u \rrp,
\]
thus $\mu_p$ is congruent modulo $p$  to $p$ times the coefficient of $u^p$ in $\zeta (\Lambda, l_t(u))$.
In order to prove  \eqref{equ_mu_E}, it is now sufficient to show that the coefficient of $u^p$ in $\zeta (\Lambda, l_t(u))$ being multiplied by $p$ becomes modulo $p$ congruent 
to $-E_{p+1}/12$. 

We firstly consider the $1/l_t(u)$ term:
\begin{multline*}
\frac{1}{l_t(u)} = \frac{1}{u} \frac{1}{1+\sum_{n \geq 1} (-1)^n\frac{u^{p^{2n}-1}}{p^n}} =  \frac{1}{u} \sum_{m \geq 0} \left(-\sum_{n \geq 1} (-1)^n\frac{u^{p^{2n}-1}}{p^n} \right)^m
\\
=  \frac{1}{u}  - \sum_{n \geq 1} (-1)^n\frac{u^{p^{2n}-2}}{p^n} + \frac{1}{u} \left( \sum_{n \geq 1} (-1)^n\frac{u^{p^{2n}-1}}{p^n} \right)^2 
- \ldots,
\end{multline*}
and we see that the formal power series of $1/l_t(u)$ has a zero coefficient of $u^p$.
Now the formal power series
\begin{multline*}
\zeta(\Lambda, l_t(u)) = \frac{1}{l_t(u)} - 2\sum_{n \geq 1} G_{2n+2} \frac{l_t(u)^{2n+1}}{(2n+1)!} \\ = 
 \frac{1}{l_t(u)} - 2\sum_{n \geq 1} G_{2n+2} \frac {\left( \sum_{m \geq 0} (-1)^m\frac{u^{p^{2m}}}{p^m} \right)^{2n+1}} {(2n+1)!} 
\end{multline*}
has the coefficient of $u^p$ (which appears only when $m=0$ and $2n+1=p$) and this coefficient is $-2G_{p+1}/p!$. We multiply it by $p$ and calculate:
\[
\mu_p \equiv -2G_{p+1}\frac{1}{(p-1)!} = \frac{B_{p+1}}{p+1} \frac{1}{(p-1)!} E_{p+1} \equiv -\frac{1}{12}E_{p+1} \pmod p
\]
for $p>3$, because $B_2/2=-1/12$, and $B_{p+1}/(p+1) \equiv B_2/2 \pmod p$ (Kummer congruences). That finishes the proof of  \eqref{equ_mu_E} for $p>3$.

We  consider the $p=3$ case separately.
We have that $G_{p+1}=G_4=(1/20)g_2$ and $E_{p+1}=E_4=-(8/B_4)G_4=12g_2$. Thus the coefficient of $u^p$ in $\zeta(\Lambda, l_t(u))$ multiplied by $p=3$ is
\[
\mu_3 \equiv -2G_{4}\frac{1}{2!}= -\frac{1}{20}g_2 \equiv -g_2 = -\frac{1}{12}E_{4} \pmod 3
\]
proving \eqref{equ_mu_E} (and $E_4 \equiv 0 \pmod 3$ since $g_2 \in \Z_3$) for $p=3$. 

\vspace{2mm}

We now turn to the proof of \eqref{equ_E_unit} in the case when $p=3$. 
Since the model \eqref{equ_ell_curve_model} has good reduction at $p=3$, the discriminant $\Delta = g_2^3 -27 g_3^2 \equiv g_2^3 \not\equiv 0 \pmod 3$, therefore $\mu_3 \equiv -g_2 \not\equiv 0 \pmod p$.

\begin{rem}
It may be interesting to note that the reduction of \eqref{equ_ell_curve_model} at $p=3$ is supersingular whenever the reduction is good.
That is well-known (see e.g. \cite[Proposition 1]{Kaneko_Zagier}).
\end{rem}

\begin{proof}[Proof of \eqref{equ_E_unit} for $p>3$]

Following \cite[Chapter X]{Lang}, we abbreviate
\[
Q=E_4, ~~ ~~ R=E_6, 
\]
and assign weights $4$ and $6$ to $Q$ and $R$ correspondingly. 
The existence of weighted polynomials $A(Q,R),B(Q,R) \in \Q[Q,R]$ with $p$-integral coefficients such that
\[
E_{p-1} = A(Q,R) ~~ \textup{and} ~~ E_{p+1} = B(Q,R)
\]
follows from \cite[Theorem X.4.2]{Lang}. Note that, while the proof of \cite[Theorem X.4.2]{Lang} depends on the $q$-expansions, we can now forget about that and think about the polynomials $A$ and $B$ as a way to calculate the values $E_{p-1}$ and $E_{p+1}$ out of the quantities $Q=12g_2$ and $R=-216g_3$. We will do that later while for the time being we consider $Q$ and $R$ as variables.
Reducing the coefficients of the polynomials $A$ and $B$ modulo $p$ and embedding the finite field $\Z/p\Z \hookrightarrow \bar{\F}_p = \F$ into its algebraic closure $\F$, we obtain the weighted polynomials $\bar{A}(Q,R), \bar{B}(Q,R) \in \F(Q,R)$. It is proved in \cite[Section X.7]{Lang} that every such polynomial splits in a unique way into a product of irreducible polynomials which are 
\begin{equation} \label{equ_P_Q_factors}
Q, ~~ R, ~~ Q^3-\alpha R^2 ~~ \textup{with $\alpha \in \F$}.
\end{equation}
Furthermore, the polynomials  $\bar{A}(Q,R)$ and $\bar{B}(Q,R)$ are relatively prime  by \cite[Theorem X.7.3(i)]{Lang}. We now plug in the modulo $p$ reductions of 
$Q=12g_2$ and $R=-216g_3$ for the variables to find that (with a slight abuse of notations)
\[
\bar{A}(12g_2,-216g_3) = 0
\]
since this is (the modulo $p$ reduction of) the Hasse invariant and \eqref{equ_ell_curve_model} is assumed to have good supersingular reduction at $p$.
It follows that the decomposition of $\bar{A}(Q,R)$  into the irreducibles involves a factor that vanishes modulo $p$ for $Q=12g_2$ and $R=-216g_3$.
Note that there may be only one such factor out of the full list \eqref {equ_P_Q_factors} since \eqref{equ_ell_curve_model} has good reduction at $p$.
Since $\bar{A}(Q,R)$ is relatively prime to $\bar{B}(Q,R)$, the decomposition of $\bar{B}(Q,R)$ into the irreducibles cannot contain that factor and we conclude that
$\bar{B}(12g_2,-216g_3) \neq 0$ in $\F$ which implies
\[
E_{p+1} = B(12g_2,-216g_3) \not\equiv 0 \pmod p
\]
as claimed.

\end{proof}



\begin{thebibliography}{99}

\bibitem{AS} Ahlgren, Scott; Samart, Detchat,
A note on cusp forms as p-adic limits. 
J. Number Theory 168 (2016), 360373.

\bibitem{Andre} Andr\'e, Yves,
Period mappings and differential equations. From $\C$ to $\C_p$.
T\^ohoku - Hokkaid\^o lectures in arithmetic geometry. With appendices by F. Kato and N. Tsuzuki. MSJ Memoirs, 12. Mathematical Society of Japan, Tokyo, 2003

\bibitem{BKY} Bannai, Kenichi; Kobayashi, Shinichi; Yasuda, Seidai, The radius of convergence of the p-adic sigma function. Math. Z. 286 (2017), no. 1-2, 751-781.

\bibitem{BK} Birch, Bryan; Kuyk, Willem (editors) Numerical tables on elliptic curves in Modular functions of one variable. IV.
Proceedings of the International Summer School on Modular Functions of One Variable and Arithmetical Applications, RUCA, University of Antwerp, Antwerp, July 17–August 3, 1972, Lecture Notes in Mathematics, Vol. 476. Springer-Verlag, Berlin-New York, 1975.

\bibitem{BFOR}  Bringmann, Kathrin; Folsom, Amanda; Ono, Ken; Rolen, Larry, Harmonic Maass forms and mock modular forms: theory and applications. American Mathematical Society Colloquium Publications, 64. American Mathematical Society, Providence, RI, 2017.

 \bibitem{BGK}  Bringmann, Kathrin; Guerzhoy, Pavel; Kane, Ben, Mock modular forms as p-adic modular forms. Trans. Amer. Math. Soc. 364 (2012), no. 5, 2393-2410. 

 \bibitem{BOR} Bruinier, Jan H.; Ono, Ken; Rhoades, Robert C., Differential operators for harmonic weak Maass forms and the vanishing of Hecke eigenvalues. Math. Ann. 342 (2008), no. 3, 673-693.
 
 \bibitem{Dicks} Dicks, Robert; Weight 2 CM newforms as p-adic limits. Ramanujan J. 58 (2022), no. 4, 1321-1332.
 
 \bibitem{Elkies} Elkies, Noam D.,
The existence of infinitely many supersingular primes for every elliptic curve over $\Q$. 
Invent. Math. 89 (1987), no. 3, 561-567
 
 \bibitem{EGO} El-Guindy, Ahmad; Ono, Ken,
Gauss's $_2F_1$ hypergeometric function and the congruent number elliptic curve.
Acta Arith. 144 (2010), no. 3, 231-239.

\bibitem{Gross} Gross, Benedict H.,
On an identity of Chowla and Selberg,
J. Number Theory 11 (1979), no. 3, S. Chowla Anniversary Issue, 344-348.
 
 \bibitem{GKO} Guerzhoy, Pavel; Kent, Zachary A.; Ono, Ken, $p$-adic coupling of mock modular forms and shadows. Proc. Natl. Acad. Sci. USA 107 (2010), no. 14, 6169-6174.

\bibitem{adele} Guerzhoy, Pavel, On Zagier's adele. Res. Math. Sci. 1 (2014), Art. 7, 19 pp.

\bibitem{Guerzhoy_P} Guerzhoy, P., On the $p$-adic values of the weight two Eisenstein series for supersingular primes, Res. Math. Sci. 12 (2025), no. 4.

\bibitem{HJ} Hanson, Michael; Jameson, Marie, Cusp forms as p-adic limits. J. Number Theory 234 (2022), 349-362

\bibitem{Hazewinkel}  Hazewinkel, Michiel, Formal groups and applications. Corrected reprint of the 1978 original. AMS Chelsea Publishing, Providence, RI, 2012

\bibitem{Hill}  Hill, Walter L., Formal groups and zeta-functions of elliptic curves. Invent. Math. 12 (1971), 321-336

\bibitem{Honda} Honda, Taira, Formal groups and zeta-functions, Osaka J. Math. 5 1968 199-213

\bibitem{Husemoller} Husem\"oller, Dale, Elliptic curves, Second edition, With appendices by Otto Forster, Ruth Lawrence and Stefan Theisen, Graduate Texts in Mathematics, 111. Springer-Verlag, New York, 2004.

\bibitem{Kaneko_Zagier} Kaneko, M.; Zagier, D.,
Supersingular j-invariants, hypergeometric series, and Atkin's orthogonal polynomials. Computational perspectives on number theory (Chicago, IL, 1995), 97-126, AMS/IP Stud. Adv. Math., 7, Amer. Math. Soc., Providence, RI, 1998

\bibitem{Katz_eis}  Katz, Nicholas M., p-adic interpolation of real analytic Eisenstein series, Ann. of Math. (2) 104 (1976), no. 3, 459-571

\bibitem{Katz_crys} Katz, Nicholas M., Crystalline cohomology, Dieudonn\'e modules, and Jacobi sums. Automorphic forms, representation theory and arithmetic (Bombay, 1979), pp. 165-246, 
Tata Inst. Fund. Res. Studies in Math., 10, Tata Inst. Fundamental Res., Bombay, 1981

\bibitem{Lang} Lang, S., Introduction to modular forms. With appendixes by D. Zagier and Walter Feit. Corrected reprint of the 1976 original. Grundlehren der Mathematischen Wissenschaften, 222. Springer-Verlag, Berlin, 1995

\bibitem{Mordell} Mordell, L. J., The congruence $((p - 1)/2)! \equiv \pm 1 \pmod p$, Amer. Math. Monthly 68 (1961), 145-146

\bibitem{Ogus} Ogus, A., A $p$-adic analogue of the Chowla - Selberg formula, $p$-adic analysis (Trento,1989), Lecture Notes in Math., 1454, Springer, Berlin, 1990

\bibitem{Robert} Robert, Alain M., A course in p-adic analysis, Graduate Texts in Mathematics, 198. Springer-Verlag, New York, 2000

\bibitem{Silverman} Silverman, Joseph H., The arithmetic of elliptic curves. Corrected reprint of the 1986 original. Graduate Texts in Mathematics, 106. Springer-Verlag, New York, 1992

\bibitem{Tajima} Tajima, Ryota, The p-adic constant for mock modular forms associated to CM forms. 
Ramanujan J. 63 (2024), no. 4, 917-929

\bibitem{Waldschmidt} 
Waldschmidt, Michel, Transcendence of periods: the state of the art, Pure Appl. Math. Q. 2 (2006), no. 2, Special Issue: In honor of John H. Coates. Part 2, 435-463.

\bibitem{Weil}  Weil, Andr\'e, Elliptic functions according to Eisenstein and Kronecker, Reprint of the 1976 original. Classics in Mathematics. Springer-Verlag, Berlin, 1999

\bibitem{Yasuda} Yasuda, Seidai, 
Explicit $t$-expansions for the elliptic curve $y^2=4(x^3+Ax+B)$. 
Proc. Japan Acad. Ser. A Math. Sci. 89 (2013), no. 9, 123-127.

\bibitem{Zagier} Zagier, D., Modular parametrizations of elliptic curves, Canad. Math. Bull. 28 (1985), no. 3, 372-384

\bibitem{PARI2}
    The PARI~Group, PARI/GP version \texttt{2.13.4}, Univ. Bordeaux, 2022,  \begin{verbatim} http://pari.math.u-bordeaux.fr/ \end{verbatim}
\end{thebibliography}
\end{document}